\documentclass[a4paper,11pt]{amsart}
\usepackage{amsmath,amsthm,amssymb,amsfonts,enumerate,color,esint,ragged2e}
\usepackage[pdftex]{graphicx}
\usepackage{setspace}
\newtheorem{thmA}{Theorem}

\newtheorem{lemA}[thmA]{Lemma}

\theoremstyle{definition}

\newtheorem{defnA}[thmA]{Definition}

\numberwithin{equation}{section}

\allowdisplaybreaks

\author[J. J. Colomina-Almi\~nana]{Juan J. Colomina-Almi\~nana}

\address{Mills College\\
Northeastern University\\
5000 MacArthur Ave., Oakland\\
CA 94613, United States of America}
\email{j.colomina-alminana@norhteastern.edu}

\author[P. R. Stinga]{Pablo Ra\'ul Stinga}

\address{Department of Mathematics\\
Iowa State University\\
396 Carver Hall, Ames\\
IA 50011, United States of America}
\email{stinga@iastate.edu}

\thanks{The order of the authors, though merely alphabetical, has been decided at random in accordance
to Boolos, that is, by flipping the coin located inside of Random's head. Research partially supported by Big 12 Faculty Fellowship Program}


\begin{document}

\title[Boolos' Hardest Logic Puzzle Ever]{Boolos' Hardest Logic Puzzle Ever \\
Can Be Solved in No Less than Three Admissible Questions: Axiomatic Framework and Rigorous Proof}

\begin{abstract}
A formal axiomatic mathematical framework for Boolos' Hardest Logic Puzzle Ever is presented
and two theorems about its solvability are proved.
By strictly following Boolos' instructions
(in particular, the requirement that all gods are always obliged to answer),
the novel concept of \textit{admissible questions} for the puzzle is introduced. 
It is then rigorously proved that Boolos' original puzzle can be solved,
in an absolute deterministic way, in no less than three yes-no admissible questions.
However, this does not mean that one could solve it in less than three admissible questions by just pure \emph{chance}.
Hence, such probabilities are computed here as well.
\end{abstract}

\maketitle

\section{Introduction}

The main aim of this paper is to introduce a new formal, axiomatic mathematical framework
for solving Boolos' Hardest Logic Puzzle Ever. The key novel concept is that of \emph{admissible question}
for the puzzle. We then rigorously prove that, within this setting, the puzzle can
be deterministically solved in no less than three admissible questions. The probabilities of solving
the puzzle by chance in less than three admissible questions are also computed.

In particular, this work emphasizes the importance of what Boolos
actually suggested as the proper way of addressing ``The Hardest Logic Puzzle Ever,'' whose origin he
attributes to Raymond Smullyan (cf.~\cite{Boolos}). Regarding this,
we remind the reader that, according to Boolos, the significance
of the puzzle is the following:

\smallskip

\begin{quote}
{\small``There is a law of logic called ``the law of excluded middle,'' according to which either $X$ is true or not-$X$ is true,
for any statement $X$ at all. (``The law of non-contradiction'' asserts that statements $X$ and not-$X$ aren't both true.)
Mathematicians and philosophers have occasionally attacked the idea that excluded middle is a logically valid law.
We can't hope to settle the debate here, but can observe that our solution to puzzle 1 made essential use of excluded middle,
exactly when we said ``Whether the middle card is an ace or not\ldots'' It is clear from The Hardest Logic Puzzle Ever,
and even more plainly from puzzle 1, that our ability to reason about alternative possibilities,
even in everyday life, would be almost completely paralyzed were we to be denied the use of the law of excluded middle.''
\cite[p.~65]{Boolos}}
\end{quote}

\smallskip
 
What this paragraph tells us is, precisely, that if one wants to obtain a valid solution according to the rules posed
by The Hardest Logic Puzzle Ever in the spirit and form professed by Boolos, one has to apply the Law of Excluded Middle.
On the other hand, if one primarily employs the Law of Non-Contradiction, some scenarios immediately
follow that invalidate a solution according to Boolos' guidelines. It thus follows that, in order to properly solve
the puzzle according to Boolos' instructions and intentions, a precise systematic framework must be established
in which the problem can be unambiguously and rigorously posed.

Our work then emphasizes that, if we were to preserve the original spirit of Boolos' work \cite{Boolos}, one
should take at heart this \textit{conservative} claim regarding the classical laws of logic to solve his puzzle.
It becomes obvious that one needs to assume the feasibility of the Law of Excluded Middle.
As it will be shown, it follows that we must restrict ourselves to a set of
\textit{admissible} questions to pose to the gods. The novel concept of admissible
question, that we introduce in this paper for the first time, is presented in Definition \ref{admissible}.
This new notion will allow us
to analyze the philosophical implications of the Hardest Logic Puzzle Ever.

Let us consider again Boolos' quote from \cite{Boolos} mentioned at the beginning.
We want to remind that the Law of Excluded Middle
tells us that in scenarios where there are three possibilities, $XYZ$, when something is identified as non-$X$, there is an open possibility that that same thing could be either $Y$ or $Z$. Keeping this in mind, let us think again about puzzle 1 in Boolos' original article:

\smallskip

\begin{quote}
{\small Puzzle 1: Noting their locations, I place two aces and a jack face down on a table, in a row; you do not see which card is placed where. Your problem is to point to one of the three cards and then ask me a single yes-no question, from the answer to which you can, with certainty, identify one of the three cards as an ace. If you have pointed to one of the aces, I will answer your question truthfully. However, if you have pointed to the jack, I will answer your question yes or no, completely at random. \cite[p.~63]{Boolos}}
\end{quote}

\smallskip
 
Besides the fact that it seems to be an election between two alternatives (Ace or Jack), Boolos points out that it is actually indifferent which is the identity of the card that was placed in the middle. This makes us believe that one should treat this as a case where one
has to find out the identities of three different individuals/tokens (Ace1, Ace2 and Jack) if one wants to have an absolute answer to the puzzle. As we said before, this is due to the fact that when the ordinary reasoning that follows from the Law of Excluded Middle
is applied, one can only find out whether the selected card in the first place is \textit{an} Ace or \textit{the} Jack, and from there one can then deduce the rest. Formally, one finds out either $J$ (which then gives options $A1$ and $A2$ as only follow-ups) or non-$J$
(which then gives $J$ and $A2$ as the only options).

We can think about this on a related case: From the fact that ``I have purchased a car painted in a primary color and my new car is not blue,'' it does not follow that the car is red, since it is a plausible possibility that the color could be yellow (if you are into this kind of flashy cars, of course).

Keeping all of this in mind, the axiomatic mathematical approach presented in this article will reinforce Boolos' \cite{Boolos} 
(and Roberts' \cite{Roberts}) original intuitions regarding the laws of logic and our ordinary ways of reasoning. To do so we deepen into some arguments that, we believe, are crucial for understanding the importance of Boolos' purpose when formulating The Hardest Logic Puzzle Ever and his three-question solution. All of this, combined,
shall help us confirm our suggestion that it is not possible to have an absolute solution to the puzzle in less than
three \emph{admissible} questions.

\section{The Hardest Logic Puzzle Ever}

George Boolos \cite[p.~62]{Boolos} presents the Hardest Logic Puzzle Ever. The puzzle goes like this (boldface is ours):

\smallskip

\begin{quote}
{\small Three gods $A$, $B$, and $C$ are called, in some order, True, False, and Random.
True always speaks truly, False always speaks falsely, but whether Random speaks truly
or falsely is a completely \textit{random} matter.
Your task is to determine the identities of $A$, $B$, and $C$ \textbf{by asking three yes-no questions};
each question must be put to exactly one god. \textbf{The gods} understand English,
but \textbf{will answer all questions} in their own language, in which the words for
``yes'' and ``no'' are ``da'' and ``ja,'' in some order. \textit{You do not know which word means which}.}
\end{quote}

\smallskip

Boolos also gives the following guidelines (boldface is ours):

\smallskip

\begin{quote}
{\small\begin{enumerate}[$(a)$]
\item It could be that some god gets asked more than one question (and hence that some god is not asked any questions at all).
\smallskip
\item \textbf{What the second question is, and to which god it is put, may depend on the answer to the first question.}
(And, of course, similarly for the third question).\smallskip
\item Whether \textbf{Random} speaks truly or not should be thought of as depending on the flip of a coin hidden in his brain:
\textbf{If the coin comes down heads, he speaks truly; if tails, falsely.}\smallskip
\item Random \textbf{will answer} ``Da'' or ``Ja'' \textbf{when asked} any yes-no question.
\end{enumerate}}
\end{quote}

\smallskip

We stress the fact that Boolos instructs us to solve the puzzle with three yes-no questions.
It is also a requirement that all the gods, including Random, will always answer.
Moreover, it is implied in Boolos' solution that the meaning of ``Da'' and ``Ja'' is irrelevant to solve the puzzle
(in other words, semantics does no determine our ontology), and, at the same time,
that all of this would be in virtue of the irreducible fundamentality of the Law of Excluded Middle.\footnote{Roberts
\cite{Roberts} reinforces Boolos' instructions and solution.
Indeed, he presents another three-question solution to the puzzle,
where the gods always reply either ``Da'' or ``Ja,'' and the use of the Law of Excluded Middle is crucial.}

We prove that Boolos' puzzle can be solved,
in an absolute deterministic way, in no less than three yes-no questions, see Theorem \ref{thm:primero}.
Towards this end, and to establish an axiomatic formal
system to unambiguously define the puzzle and prove this claim in a rigorous way,
we introduce the concept of \emph{admissible question}, see Definition \ref{admissible}.
Nevertheless, this does not mean that one could not be lucky enough to find the gods' identities
in less than three questions. It is for this reason that, second, we shall compute the probabilities of solving Boolos'
puzzle by asking no questions (!), and with one and two yes-no (admissible) questions, respectively, as well,
see Theorem \ref{thm:probabilities}.

Before presenting our setting and results, however, we need to make the following considerations
regarding Boolos' instructions that will guide and justify our analysis.

Boolos' original formulation requires that the gods ``\textbf{will answer} all questions in their own language.''
In other words, this sentence entails the rule that the gods are obliged to answer.
This in particular implies that not every yes-no question may be asked,
as it comes out of the nature of the gods themselves.
To clarify this point further, consider, for example, the question
\begin{equation}\tag{$Q$}\label{Q}
\begin{aligned}
\hbox{``Are you going to answer to this question}\quad \\ \hbox{with a word that means `no' in your language?''}
\end{aligned}
\end{equation}
This is a question that neither god True nor Random speaking Truly can answer without violating their nature
(for Random it would mean to reverse the outcome of the coin in its head).
Indeed, suppose that we ask \eqref{Q} to god True $T$ and, for the sake of simplicity,
$T$ responds in English. If $T$ replies `yes'
then $T$ is speaking falsely, because it is claiming that indeed will answer to the question with the word `no'
but it is actually saying `yes', a contradiction to its nature. On the other hand, if $T$ answers `no' then again $T$ is speaking
falsely, because it is claiming that will not answer the question with the word `no', that is, it will answer `yes',
but it is actually saying `no', a contradiction with its nature again. Therefore, $T$ will not be able to answer
to \eqref{Q} without violating its nature.
For the case of god Random speaking Truly, the same
reasoning applies. However, \eqref{Q} would be admissible as directed to god False or god Random speaking Falsely.
In fact, say we ask \eqref{Q} to god False $F$ and, for the sake of simplicity, $F$ responds in English.
By following the analysis we did before for $T$, we immediately see that there is no contradiction when  $F$ 
responds either `yes' or `no'. Both answers are lies, and that is perfectly in accordance with the nature of $F$.
Therefore, $F$ will answer either `yes' or `no' (that is, ``Da'' or ``Ja'' in its own language) to question \eqref{Q}.
Posing question \eqref{Q} to god Random speaking Falsely yields no contradiction as well.
Such a basic example shows that there must be a subclass of the class of all yes-no questions that is \textit{admissible}.
Hence we introduce the following notion.

\begin{defnA}\label{admissible}
A question $Q$ is \textit{admissible} for god $A$ if and only if, by definition, $A$ will answer either
``Da'' or ``Ja'' to $Q$.
\end{defnA}

It is important to observe that The Hardest Logic Puzzle Ever already has a
universal set of permissible questions, namely, the class of \textit{all}
yes-no questions. However, and in order to keep Boolos' spirit and instructions,
the class of possible questions must be tailored as it has to be
dynamic and precise in relation to the epistemic space.
Notice as well that the set of admissible questions does not need to be static.
In fact, Boolos remarks in guideline $(b)$ that ``What the second question is, and to which god it is put,
\textbf{may depend on the answer to the first question}'' (boldface is ours). As our proof of Theorem \ref{thm:primero}
will show, either a static or a dynamic set of admissible questions will still yield the same outcome: Three yes-no
admissible questions,
and no less, are needed to solve the puzzle with absolute certainty.
Furthermore, our proof demonstrates that the necessity of a three-admissible question solution
to deterministically solve the puzzle is independent of what we actually ask,
the meaning of ``Da'' and ``Ja,'' and the triviality of the distinction between truth-tellers and liars.

\section{The three-admissible question solution to The Hardest Logic Puzzle Ever}

As anticipated, we prove the following:

\begin{thmA}\label{thm:primero}
Boolos' Hardest Logic Puzzle Ever can be deterministically solved in no less than three admissible
yes-no questions.
\end{thmA}

\begin{proof}
According to Boolos \cite{Boolos} and Roberts \cite{Roberts}, the puzzle can deterministically be solved in
three admissible yes-no questions. Therefore, the only thing we must prove is that there are
no single-question or two-question deterministic solutions to the puzzle.

To restate the puzzle, the problem is to determine the identities of the three gods: $A$, $B$, and $C$, which are
$T=$ True, $F=$ False, and $R=$ Random, in some order. There are then 6 different possible identification scenarios,
$S1$--$S6$, see Table 1.

\begin{table*}
\begin{center}
 \begin{tabular}{||c | c | c | c||} 
 \hline
  & $A$ & $B$ & $C$ \\
 \hline\hline
 $S1$ & $T$ & $F$ & $R$ \\ 
 \hline
 $S2$ & $T$ & $R$ & $F$ \\
 \hline
 $S3$ & $F$ & $T$ & $R$ \\
 \hline
 $S4$ & $F$ & $R$ & $T$ \\
 \hline
 $S5$ & $R$ & $T$ & $F$ \\ 
 \hline
  $S6$ & $R$ & $F$ & $T$ \\
 \hline
\end{tabular}

\bigskip

\caption{Possible identification scenarios}
\end{center}
\end{table*}

First, we shall prove that there is no a single-question deterministic solution to the puzzle.
Without loss of generality, we address the first admissible yes-no question to god $A$.
With independence of what we actually ask, $A$ will always answer either ``Da'' or ``Ja.'' Independently
of what ``Da'' and ``Ja'' mean, the answer will not provide information enough to deterministically identify
the three gods at the same time.
The reason is given by the following result,
which follows at once from the fact that there is no one-to-one function from
a set of cardinality $M$ (possibilities) into a set of cardinality $N$ (answers), whenever $M>N$.\footnote{This
type of result in \cite{Wheeler-Barahona} is referred to as the ``Information Theory lemma.''}

\begin{lemA}\label{lemma}
If an admissible question has $N$ possible answers, these $N$ answers cannot distinguish $M>N$ different possibilities.
\end{lemA}

We have already specified that for our case there are 6 possible identification scenarios
($S1$--$S6$ above). Given that there are two possible answers,
$M=6$ and $N=2$. Therefore, by Lemma \ref{lemma}, we cannot distinguish the identities of all three gods altogether.

Second, to provide a two-question solution to the problem in a deterministic way,
in virtue of the same lemma, the first admissible yes-no question must reduce the remaining scenarios to no more than two.
However, the latter cannot deterministically be guaranteed. The reason is that there is always a chance that
$A$ is not-Random. In such a case, after the first question is answered, we are in a situation where
$A$ could be either True or False, so scenarios $S1$--$S4$ still remain.
Therefore, $M=4>2=N$. Notice that there is no admissible first question that
would directly determine the identity of $A$, since there are $M=3$ gods but
still $N=2$ possible answers to any admissible question. QED.
\end{proof}

In view of this proof, an important remark about the dynamics within the set of admissible question is in order.
Notice that the question \eqref{Q} we mentioned above
is not an admissible question as a first question to ask because we do not know the identity of any of the gods yet. However, it may happen that, after the first question is answered, we gain some information that will allow \eqref{Q} to become an admissible question. In this regard, our set of admissible questions can be dynamic in relation with our interaction with the gods, and not just static or absolute. Remarkably enough, our proof does not depend on the actual content of the questions that are being asked,
but depends only on the number of possible identification scenarios in relation with the number of possible answers to our questions.
To close the puzzle in a deterministic way or, which is the same, to match those numbers exactly, it does not
matter what we ask and how the set of admissible questions changes or not: We still need three yes-no admissible questions.

Theorem \ref{thm:primero} establishes that there exists a solution to Boolos puzzle in three admissible questions and no less.
Its proof is existential: It does not provide
an explicit algorithm to solve the puzzle (we do not even need that, as
algorithms were already found by Boolos \cite{Boolos} and Roberts \cite{Roberts}).
In the very interesting paper \cite{Rabern-Rabern}, Rabern and Rabern state that Boolos' original puzzle \emph{``is
no more difficult than [a] trivial puzzle''} that they present and then go
on to introduce and solve a modification of Boolos' puzzle. In particular, \cite{Rabern-Rabern}
does not provide a three-question algorithm to solve Boolos' original puzzle nor 
the ``trivial puzzle.'' However,
by using non-admissible questions like \eqref{Q} (see \cite[p.~109]{Rabern-Rabern}) and allowing
the gods to provide a third response to a question
(not only either ``yes'' or ``no'', see \cite[p.~109]{Rabern-Rabern}),
an algorithm is given for a two-non-admissible question solution
to Boolos' puzzle. In fact, \cite[p.~108]{Rabern-Rabern} has
the pre-condition that \emph{``the puzzle places no restrictions
on the type of yes-no questions to which the gods will grant an answer.''}
This is in contrast with the pre-condition introduced in this paper:
our axiomatic framework is built upon the concept of admissible questions and
the restriction to the two-response outcome of the gods inherent in Definition
\ref{admissible}. Furthermore, this approach is abstract in nature and independent
of the syntactic content of ``Da'' and ``Ja,'' and does not evaluate the puzzle in terms of ``difficulty.''
Hence, the modified and new puzzles and algorithms
presented in \cite{Rabern-Rabern} (and later on in \cite{Uzquiano} and \cite{Wheeler-Barahona}),
although very clever and interesting in their own right,
fall outside the scope of this work. On the other hand, we believe that our approach can be applied to Roberts'
puzzles found in \cite[p.~612]{Roberts}. Up to the best of our knowledge, the two new puzzles
in \cite{Roberts} and their
logical consequences have not been formally addressed in the
literature from neither an axiomatic nor a stochastic point of view.

\section{Computing the probabilities}\label{chances}

We have demonstrated that Boolos' puzzle can deterministically be solved in no less than three admissible questions.
Now, we direct our attention to the second horn of our original concern,
which is to compute the probabilities of identifying the gods without asking any question,
by asking either one or two yes-no admissible questions, respectively. This is the first time a stochastic
approach to Boolos' original puzzle has been considered.

\begin{thmA}\label{thm:probabilities}
Consider Boolos' Hardest Logic Puzzle Ever.
\begin{enumerate}[$(1)$]
\item The event
$$\big\{\hbox{solve The Hardest Logic Puzzle Ever without asking any admissible questions}\big\}$$
occurs with probability $1/6$.
\item The event
$$\big\{\hbox{solve The Hardest Logic Puzzle Ever by asking one admissible question}\big\}$$
also occurs with probability $1/6$.
\item The event
$$\big\{\hbox{solve The Hardest Logic Puzzle Ever by asking two admissible questions}\big\}$$
occurs with probability $1/3$.
\end{enumerate}
\end{thmA}

It may seem counterintuitive but, surprisingly enough,
no matter what one actually asks, the probability to find all gods' identities with only one single admissible
yes-no question is the same that we have to find them without even asking any questions at all!
Rolling a dice could save us some energy and effort.

\begin{proof}[Proof of Theorem \ref{thm:probabilities}]
For $(1)$, it is clear that there is a probability of $1/6$ of solving the puzzle without even asking a question.
Indeed, we can simply roll a 6-faces uniform fair dice, observe the outcome $X$, and then pick the corresponding scenario $SX$
from Table 1.

Let us consider $(2)$. Without loss of generality, we address the only admissible question we have available to god $A$.
There are two possibilities. Either we addressed our question to Random
(with probability $1/3$, and we are in scenarios $S5$--$S6$) or we addressed our question to a god
that is not Random (with probability $2/3$, and we are in scenarios $S1$--$S4$).
Notice that after this has happened, we have no questions left. In the first case, where $A$ is Random, there is a fifty/fifty
chance to fall under either scenario $S5$ or scenario $S6$. Hence, since we started with a probability
of $1/3$ of the question being addressed to Random, and after the first admissible yes-no question
we are left with a $1/2$ probability of deciding the correct scenario, the total probability is then
$(1/3)\times(1/2)=1/6$. In the second case, that of $A$ being non-Random with probability $2/3$, there is a
$1/4$ probability that one and only one of scenarios $S1$--$S4$ occurs. Therefore, the total probability
to find the identities of each god in one question in this second case is $(2/3)\times(1/4)=1/6$.

Finally, let us show $(3)$, namely, that there is a probability of $1/3$ to determine the gods' identities
in only two admissible yes-no questions. Independently of the first admissible question, that without loss of generality
we address to $A$, and its answer, $A$ is either
Random with probability $1/3$ or not-Random with probability $2/3$.
In the first case, that is, $A$ being Random,
we are again in scenarios $S5$ or $S6$. We can use the admissible second question to further decide in an absolute, 
deterministic way whether we are actually situated in either scenario $S5$ or scenario $S6$.
Indeed, ask god $B$ the question \textit{Does da mean yes iff Rome is in Italy?}
God $B$ will answer Da if and only if $B$ is True, and will answer Ja if and only if $B$ is False.
Therefore, the total probability for this case is $(1/3)\times1=1/3$.
In the second case, we fall under scenarios $S1$--$S4$, where we know that $A$ is not random.
By using the same admissible question as before, we identify $A$ as either True or False in a deterministic way.
This means that we can narrow down the space of epistemic possibilities to either $S1$--$S2$ or $S3$--$S4$.
Say we are situated in the space $S1$--$S2$. Then, there is a probability of $1/2$ of having either $S1$ or $S2$
as the correct identification scenario. (The same actually holds for the other space).
Hence, the total probability is $(2/3)\times1\times(1/2)=1/3$. QED.
\end{proof}

Our proof of part $(3)$ of Theorem \ref{thm:probabilities} is consistent with
the original strategy of Boolos: ``Your first move is to find a god who you can be certain is not Random'' \cite{Boolos}.
After that you ask to the not-Random god the admissible question \textit{Does Da mean yes iff Rome is in Italy?}
to decide if you are talking to god True or god False.

It is a remarkable fact that, independently of having a static or dynamic set of admissible questions,
the event $\{\hbox{identify the gods in 3 admissible questions}\}$ occurs with probability $1$,
while $\{\hbox{identify the gods in less than 3 admissible questions}\}$ has probability strictly less than $1$.
It is impossible to solve Boolos' original puzzle in two admissible questions in a deterministic way.
Notice, however, that once the event $\{\hbox{identify the gods in 2 admissible questions}\}$ has occurred
(and this would happen $1$ out of $3$ times, as part $(3)$ of Theorem \ref{thm:probabilities} showed),
the gods have definitely been identified, with absolute certainty.

One can also think about whether a yes-no question is admissible or not,
but now from a stochastic point of view. Though we believe this topic is deserving of deeper analysis, we just mention here a simple
observation. Say we consider question \eqref{Q} as a first question.
Flipping the coin in its head means that Random is speaking truly with probability $1/2$.
Then \eqref{Q} is not an admissible question not only for god True, but also for god Random.
So the probability that \eqref{Q} is an admissible question when posed to any of the gods in the first step
of our algorithm to solve the puzzle is $(1/2)\times(1/3)=1/6$. On the other hand, Random speaks falsely with probability $1/2$.
In this case \eqref{Q} is not admissible only for god True. Thus the probability that \eqref{Q}
is admissible in the first step of a solution algorithm in this case is $(1/2)\times(2/3)=1/3$.
This problem will be addressed elsewhere.

\section{Conclusion}

This paper introduced the novel concept of admissible question for Boolos' Hardest Logic Puzzle Ever
and demonstrated via a formal, abstract mathematical argument, that the puzzle can
be deterministically solved by using no less than three admissible questions.
Deciding whether or not finding such questions --an algorithm-- has no effect in the
analysis nor in the proof. Furthermore, a new stochastic approach to Boolos' puzzle is considered and
the probabilities of solving the puzzle with less than three
admissible questions were computed. We believe that this formal mathematical setting
will be suited to rigorously interpret and analyze Robert's puzzles \cite[p.~612]{Roberts}.

Therefore, it has been demonstrated that a proper understanding of Boolos'
original purpose when proposing his puzzle and a correct interpretation of the
statement and its instructions are crucial to providing an adequate answer.
The satisfactory solution to the puzzle requires
considering the natural issue of the admissibility of the questions posed to the gods,
which definitively requires taking the inferential powers of the Law of Excluded Middle into account.
As mentioned above, addressing properly the puzzle necessitates close attention to Boolos'
requirements and intentions, specifically to the claim that the Law of Excluded Middle is a more adequate
explanatory source for certain ways of reasoning than the Law of Non-Contradiction.
Denying this might align any potential solution to an intuitionistic position,
potentially jeopardizing the efficacy of inductive reasoning, and
perhaps even falling short in adequately answering
questions regarding the sources of some of our beliefs. After all, as Boolos mentions
at the end of his now-classical article: \emph{``Mathematicians and philosophers have occasionally
attacked the idea that excluded middle is a logically valid law [\ldots] It is clear from The Hardest Logic Puzzle Ever
[\ldots] that our ability to reason about alternative possibilities, even in everyday life, would be almost completely
paralyzed were we to be denied the use of the law of excluded middle.''} \cite[p.~65]{Boolos}

This is the main difference between what the present article evidences as the adequate answer
within the axiomatic framework of admissible questions
to the puzzle proved in Theorem \ref{thm:primero} and any other potential solution, mostly any solution invested in reducing the richness
of our logical inferential rules to a unique logical mechanism, such as \textit{Reductio ad Absurdum}.



\end{document}